\documentclass[12pt]{amsart}
\usepackage{amsmath,amscd,amsthm,amsfonts, amssymb,amsxtra}
\usepackage[all]{xy}
\SelectTips{cm}{}
\newtheorem{thm}{Theorem}[section]  
\newtheorem*{un-no-thm}{Theorem}
\newtheorem{cor}[thm]{Corollary}     
\newtheorem{lem}[thm]{Lemma}         
\newtheorem{prop}[thm]{Proposition}  

\newtheorem{bigthm}{Theorem}

\newtheorem{bigcor}[bigthm]{Corollary}

\theoremstyle{definition}

\theoremstyle{definition}

\theoremstyle{definition}
\newtheorem*{ques}{Question}

\theoremstyle{remark}
\newtheorem{rem}[thm]{Remark}        

\newtheorem*{acks}{Acknowledgements}
\newtheorem*{out}{Outline}

\newtheorem*{rems_un}{Remarks}

\newtheorem*{exs}{Examples}

\DeclareMathOperator{\tr}{tr}

\begin{document}
\title[Poincar\'e duality and Periodicity, II.]{Poincar\'e duality and Periodicity, II.\\ James Periodicity}
\date{\today}
\author{John R. Klein}
\address{Dept.\ of Mathematics, Wayne State University, Detroit, MI 48202}
\email{klein@math.wayne.edu}
\author{William Richter}
\address{Dept.\ of Mathematics, Northwestern University, Evanston, IL 60208}
\email{richter@math.northwestern.edu}
\dedicatory{To Bruce Williams on his sixtieth birthday}
\thanks{The author is partially supported by the NSF}
\begin{abstract}
Let $K$ be a connected finite complex. This paper studies the problem of whether one can
attach a cell to some iterated suspension $\Sigma^j K$ so that the resulting
space satisfies Poincar\'e duality. When this is possible, we say that $\Sigma^j K$ is a {\it spine.} 

We introduce the notion of {\it quadratic self duality} and show that if 
$K$ is quadratically self dual, then  
$\Sigma^j K$ is a spine whenever $j$ is a suitable power of  two.
The powers of two come from the
James periodicity theorem.
  
We briefly explain how
our main result, considered up to bordism, gives a new interpretation  of the four-fold
periodicity of the surgery obstruction groups. We therefore obtain a relationship between James
periodicity and the four-fold periodicity in $L$-theory.
\end{abstract}

\maketitle
\setlength{\parindent}{15pt}
\setlength{\parskip}{1pt plus 0pt minus 1pt}
\def\bdot{\bold .}
\def\Top{\bold T\bold o \bold p}
\def\Sp{\bold S\bold p}
\def\vo{\varOmega}
\def\smsh{\wedge}
\def\flush{\flushpar}
\def\id{\text{\rm id}}
\def\dbslash{/\!\! /}
\def\codim{\text{\rm codim\,}}
\def\:{\colon}
\def\holim{\text{holim\,}}
\def\hocolim{\text{hocolim\,}}
\def\cal{\mathcal}
\def\Bbb{\mathbb}
\def\bold{\mathbf}
\def\simtwohead{\,\, \hbox{\raise1pt\hbox{$^\sim$} \kern-13pt $\twoheadrightarrow \, $}}
\def\codim{\text{\rm codim\,}}
\def\stableto{\mapstochar \!\!\to}

\section{Introduction} This paper is extends the results of \cite{Klein-Richter}.
Recall that the {\it spine} of a Poincar\'e duality space $X$ of dimension $d\ge 3$
is a finite CW complex $K$ of dimension $\le d-1$
 such that $X$ is homotopy equivalent to  $K \cup_{\beta} D^d$,
in which $\beta\:S^{d-1} \to L$ is the attaching map for the
top cell of $X$ (\cite[2.4]{Wall3}, \cite{Wall}).  The pair $(K,\beta)$ is 
unique up to homotopy type.

We shall consider here
the problem of deciding when a finite complex is stably the spine of some Poincare space:

\begin{ques}
Let $K$ be a connected based finite complex. When is some $j$-fold suspension $\Sigma^j K$
the spine of some Poincare space $X$?
\end{ques}

A question like this appears 
in the problems section of the proceedings of the 1982 Northwestern homotopy theory conference
\cite{Hodgson}. It is not hard to construct spaces $K$ answering the question in the negative. Nevertheless, we will see that the question does have a solution for 
certain interesting class of spaces. Roughly, these spaces will be 
``very self Spanier-Whitehead dual'' in a sense that we will explain.

Our first main theorem says that if $K$ is already a spine of an $X$ 
and the latter has trivial Spivak fibration, then $\Sigma^j K$ for infinitely many $j$
are also spines. More precisely, let us call a Poincar\'e space $X = K \cup_\beta D^d$
{\it stably parallelizable} if its Spivak normal fibration is trivializable. This is equivalent to the
statement that the attaching map $\beta\: S^{d-1} \to K$ is stably null homotopic, and also 
to the statement that the top cell of $X$ stably falls off.

Let us recall 
that the {\it Adams number} $\phi(\ell)$ is the number of of positive integers $n\le \ell$
such $n$ is congruent to $0,1,2,4$ modulo $8$.

\begin{bigthm} \label{first} 
Assume $\ell \ge 2$. Suppose $K$ is the spine of a Poincar\'e space $X^d$ whose 
top cell falls off after $\ell$ suspensions. 

Then there is a Poincar\'e space $Y^{d+2j}$ with spine
$\Sigma^jK$, where $j = 2^{\phi(\ell-1)}$. Furthermore, the top cell of $Y$ splits off
after $\ell$ suspensions.
\end{bigthm}

\begin{rems_un} (1). The proof will show that
 $Y$ is functorial in $\beta$ and the choice of null homotopy of $\Sigma^\ell \beta$.
 \smallskip
 
 {\flushleft (2).} Although they have different dimensions,  $X$ and $Y$ 
 have the same Euler characteristics and signatures. If
 we ignore gradings, then the intersection pairings of $X$ and $Y$
 are isomorphic. 
\smallskip

{\flushleft (3).} In a previous paper \cite{Klein-Richter} we considered the
case $\ell = 1$ and showed that the analogous statements are true provided one takes $j = 2$.
\smallskip

{\flushleft (4).} One can apply the theorem iteratively to produce a sequence of Poincar\'e
spaces $X_s$ with spine $\Sigma^{sj} X$. 
\end{rems_un}

If $K$ fails to be a spine, then one can still give a partial solution to the basic question
in terms of Spanier-Whitehead duality.
Let $$D_2 K \,\, := \,\, (E\Bbb Z_2)_+ \smsh_{\Bbb Z_2} K^{\smsh 2} $$ be the {\it quadratic construction}, i.e., the reduced homotopy orbits of $\Bbb Z_2$ acting by transposition on the smash product $K \smsh K$. Note that $D_2$ is a functor on the
category whose objects are based spaces and whose morphisms are stable maps.
In what follows we use the notation
$A\stableto  B$ to denote a stable map of spaces from $A$ to $B$.

A stable map
$$
S^d \stableto D_2 K
$$
is said to be a {\it quadratic duality} if the composition
$$
S^d \stableto D_2 K \overset{\tr} \stableto K \smsh K 
$$
is an $S$-duality map, where the second displayed map is  the
 transfer associated with the action of $\Bbb Z_2$ on $K\smsh K$. 
(Recall that an $S$-duality  $u\: S^n \stableto A\smsh B$ is a stable map
such that the operation 
$f\mapsto (f\smsh {\text{id}}_B)\circ u$ gives an isomorphism $\{A,E\} \cong \{S^n, E\smsh B\}$, where
$\{A,E\}$ is the abelian group of homotopy classes of stable maps from $A$ to $E$.) Our notion of quadratic duality is
a space level version of a construction given by Ranicki in the category of chain complexes \cite{Ranicki}.

\begin{bigthm} \label{quadratic} The following statements are equivalent.
\begin{enumerate} 
\item  There exists a quadratic duality map $S^d \stableto D_2K$.
\item There exists a stably parallelizable 
Poincar\'e duality space $X^{d+2j}$ with spine $\Sigma^j K$ for some
non-negative integer $j$.
\end{enumerate}
\end{bigthm}

 Theorem \ref{quadratic} can be refined to give 
 explicit information about the integers $j$ for which the statement (2) holds. 
 We say that that $K$ is 
{\it $j$-periodic} for some positive integer $j$ if there
is an integer $c$ and a sequence of Poincare complexes $X_1,X_2, \dots$ 
such that the spine of $X_i$ is $\Sigma^{c+ij}K$. If no such pair exists, call $K$
{\it aperiodic}.

\begin{exs} (1). The sphere $S^k$ is aperiodic:
since there are only finitely many homotopy classes having Hopf invariant one, 
there are only finitely many complexes 
of the form $S^p \cup_\beta D^n$ satisfying Poincar\'e duality.

{\flushleft (2).} The wedge $S^k \vee S^k$ is $1$-periodic, since $
S^{j} \vee S^{j}$ is the spine of $S^j \times S^j$ for every $j$. 
\smallskip

{\flushleft (3).} The spine of the Stiefel manifold $V_2(\Bbb R^{2j+3})$
of two-frames in $\Bbb R^{2j+3}$ is $\Sigma^{2j}\Bbb RP^2$. Consequently, $\Bbb RP^2$ is
$2$-periodic.  It is known not to be $1$-periodic, even though $\Sigma \Bbb RP^2$ is the
spine of a five manifold (cf.\ \cite{Klein-Richter}).
\smallskip
\end{exs}

Let 
$$
D_2^\ell K  = S^{\ell-1}_+ \smsh_{\Bbb Z_2} K^{\smsh 2}
$$
where $\Bbb Z_2$ acts antipodally on $S^{\ell-1}$ (observe that $D_2^\infty K = D_2 K$ contains
$D_2^\ell K$).

An {\it $\ell$-refined quadratic duality} for $K$ is stable map 
$$
S^d \stableto D_2^\ell K
$$
such that the composite
$$
S^d \stableto D_2^\ell K \overset{\subset}\to D_2 K 
$$
is a quadratic duality. Since the  the inclusion
$D_2^\ell(K) \to D_2(K)$ is $\ell$-connected, it follows that any quadratic
duality factors through an $\ell$-refined one when $\ell$ is sufficiently large. In fact, if $K$ is $r$-connected, then 
$\ell$ can be taken to be $d-2r-1$.

\begin{bigthm} \label{quadratic-refinement} Assume $\ell \ge 2$. Then the following statements are equivalent
\begin{enumerate} 
\item There is an $\ell$-refined quadratic duality $S^d \stableto D_2^\ell K$.
\item There is a Poincar\'e space $Y^{d+2j}$ with spine $\Sigma^j K$ and whose
top cell falls off after $\ell$-suspensions, where
$j$ is some non-negative integer multiple of $2^{\phi(\ell-1)}$. 
\end{enumerate}
\end{bigthm}

If we combine Theorems \ref{first} and \ref{quadratic-refinement} we obtain

\begin{bigcor} Assume $\ell \ge 2$. Then
the complexes $K$ satisfying Theorem 
\ref{quadratic-refinement} are $2^{\phi(\ell-1)}$-periodic.

If $\ell = 1$, then such $K$ are $2$-periodic.
\end{bigcor}

\begin{bigcor} \label{2-period} Assume $K$ is $r$-connected, $d = 2r+2$ and there is 
a quadratic duality $S^d \stableto D_2(K)$. Then $K$ is $2$-periodic. 
\end{bigcor}

\subsection*{The relationship with surgery theory}
Let $f\: X^d \to S^d$ be a Poincar\'e normal map, 
that is, the pullback of the stable normal bundle of $S^d$ is identified with
the Spivak normal bundle of $X$; this is the same thing as equipping $X$ with a stable 
fiber homotopy trivialization $b$ of
its Spivak normal fibration. 
Poincar\'e normal maps of this kind arise from stably parallelizable Poincar\'e spaces
$X = K \cup D^n$ by collapsing $K \subset X$ to a point. 
Up to homotopy, every normal map $f\: X \to S^d$ is represented
in this way. 

One is typically interested in the
obstruction to deciding when $(f,b)$ is normally cobordant to a homotopy equivalence.
If $d$ is odd, then there is no such obstruction (this uses \cite{Klein_haef}). 

If $d = 2s$ is even, 
the surgery obstruction 
$\sigma(f,b)$ is defined and lives in the $L$-group $L_d = \Bbb Z/(1+(-1)^{s-1})$.
Poincare embedding theorems for spheres \cite{Klein_compress} can be used to show
that the surgery obstruction is the complete obstruction when $d \ge 6$.
A version of Wall realization also shows that any element of $L_d$ comes from
a Poincar\'e normal map. Consequently, the bordism group of normal maps
$X^d \to S^d$ is in bijection with $L_d$ when $d = 2s \ge 6$
(this statement is implicit in \cite{HV}; I intend to publish a homotopy theoretic
proof of it in the near future).

Furthermore, one can do surgery to represent any element of
the normal bordism group by a normal map 
$f\: X^d \to S^d$ which is $s$-connected. The spine $K$ of $X$ is then 
$(s-1)$-connected, and by Poincar\'e duality, it has the homotopy type of a CW complex of dimension $\le s$. Consequently, $K$ has the homotopy type of a finite wedge of $s$-spheres, and so,
using Corollary \ref{2-period}, one sees that $K$ is $2$-periodic. We therefore obtain a sequence of
stably parallelizable Poincar\'e spaces $\{X^{d+4j}\}_{j\ge 0}$ such that the
spine of $X^{d+4j}$ is $\Sigma^{2j} K$. 

Consider the associated Poincar\'e normal maps $(f_j,b_j)\: X^{d+4j} \to S^{d+2j}$. These have surgery obstructions $\sigma(f_j,b_j) \in L_{d+4j} = L_d$. It can be seen from the constructions appearing in our proof that all of these surgery obstructions coincide. It follows that periodicity in the sense of this paper corresponds to the four-fold periodicity in surgery theory. However, our interpretation
of the periodicity is different from the usual one: the periodicity operator for $L$-theory
is usually described by taking cartesian product with $\Bbb C P^2$.

\begin{out} Section 2 gives a criterion for recognizing 
Poincar\'e duality on those spaces whose top cell falls off. Section 3  discusses James periodicity.
Section 4 describes a $\Bbb Z_2$-equivariant version of the Whitehead product map. Section 5 contains the proof of Theorem \ref{first}. Section 6 proves Theorems \ref{quadratic} and \ref{quadratic-refinement}. Section 7 is an appendix proving a composition formula for the stable Hopf invariant.

We will be working in the category of based spaces having the homotopy type of a based CW complex.
A stable map from $X$ to $Y$ is a map $\Sigma^\infty X \to \Sigma^\infty Y$ of associated suspension
spectra. We will use the notation $X\stableto Y$ to denote such maps.

We use the notion of Poincar\'e complex defined by Wall \cite{Wall3}:  $X$ is said to be a Poincar\'e duality space of dimension $d$ if there is a pair $([X],\cal L)$ consisting of a rank one local coefficient system on $X$ and a class $[X] \in H_d(X;\cal L)$ such that 
the cap product homomorphism $\cap [X]\: H^*(X;M) \to H_{d-*}(X;\cal L\otimes M)$ is
an isomorphism, where $M$ ranges over all systems of local coefficients.
\end{out}
 
\begin{acks} We are indebted to Andrew Ranicki for $L$-theory related discussions, to  
Ralph Cohen clarifying James periodicity and to  
Fred Cohen  for discussions surrounding  the equivariant Whitehead product. 
 
\end{acks}

\section{A criterion for Poincar\'e duality}

Let $X$ be a space of the form  $K \cup_\beta D^d$, for some attaching map 
$\beta\: S^{d-1} \to K$. We will assume that $K$ is a connected, based CW complex of
dimension $< d$. 
 
Assume $\beta$ is stably null homotopic. A choice of null homotopy $\eta$
then amounts to map of pairs
$$
(D^d,S^{d-1}) \to (QK,K)
$$
where $QK = \Omega^\infty \Sigma^\infty K$. Let
$$
h\:QK \to QD_2K
$$
be the stable Hopf invariant \cite[p.\ 62]{Crabb} . As the composite $K \to QK \to QD_2K$ has a preferred null homotopy,
$h$ actually determines a map of pairs
$$
(QK,K) \to (QD_2K,CK)
$$
where $CK$ is the reduced cone on $K$. We will also denote this map of pairs by $h$. Taking the composite, we have a map 
$$
(D^d,S^{d-1}) \to (QK,K) \to (QD_2K,CK)
$$
which determines a homotopy class of map of quotients
$$
S^d \to QD_2K \, ,
$$
or equivalently a homotopy class of stable map
$$
\delta(\beta,\eta)\: S^d \stableto D_2K\, .
$$

\begin{prop} \label{criterion} The space $X$ is Poincar\'e duality space of dimension $d$
if and only if the composite
$$
\tr_K \circ \delta(\beta,\eta) \: S^d \stableto K \smsh K
$$
is an $S$-duality map.
\end{prop}

\begin{proof} 
The pair $\beta$ and $\eta$ determine a stable map $\rho\: S^d \stableto X_+$
and stable retraction $\phi\: X_+ \stableto K$. Then $X$ will satisfy Poincar\'e duality
if and only if the composite
$$
S^d \overset\rho \stableto X_+ \overset\Delta\to X_+ \smsh X_+
$$
is an $S$-duality. It is known that the latter coincides with
composite
$$
\tr_{X_+} \circ h(\rho) \: S^d \stableto D_2(X_+) \stableto X_+ \smsh X_+ 
$$
(cf.\ \cite{Ranicki}, \cite[p.\ 62]{Crabb}).
Observe that $K$ is the effect of removing the bottom and top cells of $X_+$, and these
cells are dual to each other.
Consequently, Poincar\'e duality for $X$ is equivalent to
$S$-duality for $(\phi \smsh \phi) \circ \tr_{X_+} \circ h(\rho)$.

Observe also that $\phi$ is compatible with transfers, in the sense that
$$
(\phi \smsh \phi) \circ \tr_{X_+} \,\, \simeq \,\, \tr_K \circ D_2(\phi)\,.
$$ Hence, 
Poincar\'e duality holds for $X$ if and only if 
$\tr_K \circ D_2(\phi) \circ h(\rho)$ is an $S$-duality.

If we apply the composition formula for $h$ (cf.\ Appendix) applied to the null
homotopic stable map $\phi \circ \rho$, we see that
$$
D_2(\phi) \circ h(\rho) = - h(\phi) \circ \rho \, .
$$
Consequently,  $\tr_K \circ D_2(\phi) \circ h(\rho)$ coincides with 
$\tr_K \circ h(\phi) \circ \rho$ up to sign. We will be done if we 
can show that the stable composite
$$
h(\phi) \circ \rho \: S^d \stableto X_+ \stableto D_2(K)
$$
coincides with $\delta(\beta,\eta)$. The point is that $h(\phi)$
restricted to $K_+$ is null homotopic in a preferred way, so we have a factorization
of $h(\phi) \circ \rho$ as
$$
S^d \stableto X_+ \to X_+/K_+ = D^d/S^{d-1} \stableto D_2K\, .
$$
The composite of the first two of these maps is clearly the identity .
Finally,  the last map $D^d/S^{d-1} \stableto D_2K$ coincides with $\delta(\beta,\eta)$.
\end{proof}

\section{James Periodicity}

The quotient
$$
\Bbb RP_m^n := \Bbb RP^n/\Bbb RP^{m-1}\, , 
$$
is standard notation for truncated real projective space.
The periodicity theorem of James  states
$$
\Bbb RP_{m+j}^{n+j} \,\, \simeq  \,\, \Sigma^j \Bbb RP_m^n
$$
whenever $j$ is a suitable power of two \cite{James}, \cite{Mahowald}. 
We will be primarily interested in the case $m =0$. The proof in this instance
follows from the computation of the order of the canonical line bundle 
in the reduced real $K$-theory of $\Bbb RP^n$. We will reformulate the precise
statement below.

Let $\alpha$ denote the involution of $\Bbb R^1$ given by $t\mapsto -t$; this is
sign representation of $\Bbb Z_2$, and is sometimes denoted by $\Bbb R^\alpha$.
Let $S(\ell \alpha)$ be the unit sphere of $\Bbb R^{\ell\alpha}$; this
is just $S^{\ell -1}$ with antipodal action.
Let $S^{\ell \alpha}$ be the one point compactification of $\Bbb R^{\ell\alpha}$.

The total space of the
canonical line bundle $\xi_\ell$ over $\Bbb R P^{\ell -1}$ is 
$$
S(\ell \alpha) \times_{\Bbb Z_2} \Bbb R^\alpha\, ,
$$
where bundle projection to $\Bbb R P^{\ell-1}$ is given by first factor projection. By
Adams \cite{Adams}, the order of $\xi_\ell$ is $j:= 2^{\phi(\ell - 1)}$, in the sense that
$j$ is the smallest positive integer for which $j\xi_\ell$ is stably trivial. 

If we write $\ell - 1 = 8d +r$, where $0\le r < 8$, then 
$$
j = \phi(8d+r) \ge \phi(8d) = 2^{4d}
> 8d + r = \ell - 1\, .
$$ It follows that $j\xi_\ell$ is trivial, since it is a stably trivial rank 
$j$-bundle over a CW complex of dimension $< j$. In particular, there is an isomorphism of fiberwise
one point compactified sphere bundles
$$
S(\ell \alpha) \times_{\Bbb Z_2} S^{j\alpha} \,\, \cong \,\, \Bbb P^{\ell -1} \times S^j
$$
over $\Bbb P^{\ell-1}$. This isomorphism preserves zero sections. 

Suppose $A$ is a based $\Bbb Z_2$-space. Then $S(\ell\alpha) \times_{\Bbb Z_2} A \to \Bbb P^{\ell-1}$
is a fibration with section. The isomorphism of sphere bundles above induces, by taking fiberwise smash products, an equivalence of fibrations with section
$$
S(\ell \alpha) \times_{\Bbb Z_2} (S^{j\alpha}\smsh A)  \,\, \simeq \,\, 
S(\ell\alpha) \times_{\Bbb Z_2} (\Sigma^j A)\, .
$$
Collapsing zero sections, we obtain 

\begin{lem} \label{pre-James} Assume $ \ell \ge 2$.
For $j = 2^{\phi(\ell-1)}$ and a based $\Bbb Z_2$-space $A$,
there is a weak equivalence of based spaces
$$
S(\ell\alpha)_+ \smsh_{\Bbb Z_2} (S^{j\alpha}\smsh A)  \,\, \simeq  \,\,
S(\ell\alpha)_+ \smsh_{\Bbb Z_2} (\Sigma^j A)\, .
$$
\end{lem}

\begin{rem} If we take $A = S^0$, the statement gives a homotopy equivalence
$
\Bbb RP^{\ell+j-1}_{j-1}  \simeq  \Sigma^j\Bbb RP^{\ell-1}
$.
\end{rem}

\begin{cor} \label{James} With $\ell$ and $j$ as above and 
$K$ a based space, there is a weak equivalence
$$
\Sigma^{j-1}(S(\ell\alpha)_+ \smsh_{\Bbb Z_2} (S^{j\alpha} \smsh K^{\smsh 2}) 
\,\, \simeq \,\, \Sigma^{2j-1} D_2^\ell(K)\, .
$$
After suspending once, this becomes an equivalence 
$$
D^\ell_2(\Sigma^j K) \,\, \simeq \,\, \Sigma^{2j} D^\ell_2(K) 
$$
which is compatible with transfers, in the sense
that the diagram
$$
\xymatrix{
D^\ell_2(\Sigma^j K)   \ar[r]^\simeq \ar[d]_{\tr_{\Sigma^j K}}  & \Sigma^{2j} D_2^\ell(K) 
\ar[d]^{\Sigma^{2j}\tr_K} \\
 \Sigma^j K \smsh \Sigma^j K \ar[r]_{=} & \Sigma^{2j} K \smsh K
}
$$
homotopy commutes.
\end{cor} 

\begin{proof} For the first part, use take $A = \Sigma^{j-1} K \smsh K$ with evident action and
apply \ref{pre-James}. As for the second part, the equivalence
$D^\ell_2(\Sigma^j K) \simeq \Sigma^{2j} D_2^\ell(K)$ arises from an isomorphism of bundles over
${\Bbb P}^{\ell-1}$ which can be pulled back along $S(\ell \alpha)$ to obtain an isomorphism of 
double covers of the total spaces. Compatibility with transfers now follows.
\end{proof}

\begin{rem} When $\ell = 1$ we can take $j =2$ and the conclusions  of
\ref{James} will hold.
\end{rem}

\section{A generalized Whitehead product}

The Whitehead product $w\:\Sigma K \smsh K \to \Sigma K$
is obtained from the cofiber sequence
$$
\begin{CD}
\Sigma K \smsh K @> W >>  \Sigma K \vee \Sigma K @> \subset >> \Sigma K \times \Sigma K
\end{CD}
$$
by taking the composition of $W$ with the fold map.

We will construct below an equivariant (weak) map
$$
\tilde \omega\: S^\alpha \smsh  (K \smsh K) \to \Sigma K
$$
where the target has trivial involution. Unequivariantly, this map
will coincide with $w$ up to homotopy.

\subsection*{Construction of $\tilde\omega$}  Recall that the {\it join} $K \ast K$ is the quotient space
of $K \times K \times [0,1]$ in which  $(x,y,0)$ and $(x',y,0)$ are identified, and
also $(x,y,1)$ and $(x,y',1)$ are identified for all $x,x',y,y' \in K$. The map
$(x,y,t) \mapsto (t,(x,y))$ defines a homotopy equivalence
$$
K \ast K \overset \sim \to \Sigma K \smsh K \, .
$$
If we give $K \ast K$ the involution $(x,y,t) \mapsto (y,x,1-t)$ and 
$S^\alpha \smsh K \smsh K$ the involution $(t,(x,y)) \mapsto (1-t,(y,x))$, then the displayed map is
is equivariant, so we have an equivariant weak equivalence. We henceforth identify
$$
K \ast K \,\, \simeq \,\, S^\alpha \smsh K \smsh K \, .
$$

Consider the map 
$$
K \ast K \to \Sigma (K \vee K)
$$
defined by the formula
$$
(x,y,t) \mapsto
\begin{cases}
(4t, x,*),  \quad & t \in [0,1/4],\\
(4t-1,*,y),  & t \in [1/4,1/2],\\
(3-4t,x,*), & t \in [1/2,3/4], \\
(4-4t,*,y), & t \in [3/4,1].
\end{cases}
$$
If we give $\Sigma (K \vee K)$ the involution which is trivial
on the suspension coordinate, but transposes summands (i.e., $(t,(x,y)) \mapsto (t,(y,x))$)
then the map is $\Bbb Z_2$-equivariant. If we compose with the fold map, we get
a $\Bbb Z_2$-equivariant map
$$
K \ast K \to \Sigma K\, ,
$$
which we will take the liberty of writing as an 
equivariant map
$$
\tilde \omega\:  S^\alpha \smsh (K \smsh K) \to \Sigma K \, ,
$$
even though the map $K \ast K \to \Sigma K$ does not obviously factor
through the equivariant weak equivalence $K \ast K \to S^\alpha \smsh (K \smsh K)$.

\subsection*{The map $\omega_\ell$}
For each $\ell \ge 1$,
we will define a map 
$$
\omega_\ell \: S(\ell \alpha)_+ \smsh_{\Bbb Z_2} S^\alpha \smsh (K \smsh K) 
\to \Sigma K \, .
$$
Let $S(\ell \alpha)_+ \to S^0$ be the map which collapses $S(\ell\alpha)$ to a
point. Smash this map with $\tilde \omega$. This results in a $\Bbb Z_2$-equivariant map
$$
S(\ell \alpha)_+ \smsh  (S^\alpha \smsh (K \smsh K)
\to \Sigma K \, .
$$
Finally, $\omega_\ell$ is the
factorization of this last map over orbits; here we
are using the fact that  target has trivial involution.

Let
$$
E_\ell \: \Sigma K \to \Omega^\ell \Sigma^\ell (\Sigma K)
$$
be the  adjoint to the identity.

\begin{lem} \label{EP=0} The composite $E_\ell\circ \omega_\ell$
has a preferred null homotopy. Furthermore, if $C$ denotes the cone on 
$S(\ell\alpha)_+\smsh_{\Bbb Z_2} (S^\alpha \smsh K^{\smsh 2})$, then 
the associated commutative square
$$
\xymatrix{
S(\ell\alpha)_+\smsh_{\Bbb Z_2} (S^\alpha \smsh K^{\smsh 2})
 \ar[r]^(.7){\omega_\ell} \ar[d]_{\cap} &
\Sigma K \ar[d]^{E_\ell}\\
C \ar[r] & \Omega^\ell \Sigma^\ell (\Sigma K)
}
$$
is $(4r+7)$-cocartesian when $K$ is $r$-connected.
\end{lem}

\begin{proof} The second filtration  $F_2 (\Sigma K)$ of the configuration space
model for
$\Omega^\ell \Sigma^\ell (\Sigma K)$ sits in a pushout square
$$
\xymatrix{
S(\ell\alpha)_+ \smsh_{\Bbb Z_2} (\Sigma K)^{\vee 2}  \ar[r] \ar[d] &
S(\ell\alpha)_+ \smsh_{\Bbb Z_2} (\Sigma K)^{\times 2} \ar[d] \\
\Sigma K \ar[r] & F_2(\Sigma K) 
}
$$
in which the horizontal arrows are inclusions and the left vertical arrow
is the evident fold map (\cite{May}, \cite{Snaith}).
The top horizontal arrow sits in a cofiber sequence
$$
S(\ell\alpha)_+ \smsh_{\Bbb Z_2} (S^\alpha \smsh K^{\smsh 2}) \to
S(\ell\alpha)_+ \smsh_{\Bbb Z_2} (\Sigma K)^{\vee 2}\to
S(\ell\alpha)_+ \smsh_{\Bbb Z_2} (\Sigma K)^{\times 2}  
$$
where the first map is given by smashing $\tilde\omega$ with the identity of
$S(\ell \alpha)_+$ and then taking orbits. It follows that there is also
a cofiber sequence
$$
S(\ell\alpha)_+ \smsh_{\Bbb Z_2} (S^\alpha \smsh K^{\smsh 2})
\to \Sigma K \to F_2(\Sigma K)\, .
$$
Moreover, the composite
$$
S(\ell\alpha)_+ \smsh_{\Bbb Z_2} (S^\alpha \smsh K^{\smsh 2})
\to \Sigma K \to F_2(\Sigma K) \to \Omega^\ell\Sigma^\ell (\Sigma K)
$$
coincides with $E_\ell\circ \omega_\ell$,  and is therefore null-homotopic. The other part of
the lemma also follows from the observation that
the map $F_2(\Sigma K) \to \Omega^\ell\Sigma^\ell (\Sigma K)$
is $(4r+7)$-connected.
\end{proof}

\begin{cor}\label{useful} Let $j = 2^{\phi(\ell-1)}$, and
 use $\Sigma^j K = \Sigma^{j-1} (\Sigma K)$ in place of $\Sigma K$.
Then using \ref{James}, the map $\omega_\ell$ may be rewritten in the form
$$
\Sigma^{2j-1} D_2^\ell{K} \to \Sigma^j K \, .
$$
\end{cor}

\section{Proof of Theorem \ref{first}}

Let $X = K \cup_\beta D^d$, and suppose $\Sigma^\ell \beta: S^{d+\ell-1} \to 
\Sigma^\ell K$ comes equipped with a null homotopy $\eta$. Then we have a map of pairs
$$
\phi(\eta,\beta)\: (D^d,S^{d-1}) \to (\Omega^\ell\Sigma^\ell K, K)
$$
Let $E_\ell\: K \to \Omega^\ell \Sigma^\ell K$ be adjoint to the identity,
and let
$$
h_\ell \: \Omega^\ell\Sigma^\ell K \to \Omega^\ell\Sigma^\ell D_2^{\ell}K
$$
be the Hopf invariant. Then $h_\ell\circ E_\ell$ is preferred null homotopic, so
$h_\ell$ may be considered as a map of pairs
$$
 (\Omega^\ell\Sigma^\ell K, K) \to (\Omega^\ell\Sigma^\ell D_2^\ell K, CK)\, .
 $$
 Then by taking quotients $h_\ell\circ \phi(\eta,\beta)$ induces a map
 $$
 f\: S^d \to \Omega^\ell\Sigma^\ell D_2^{\ell}K \, .
 $$
 By the material in \S2, the composite
 $$
 \begin{CD} 
 S^d @> f >> \Omega^\ell\Sigma^\ell D_2^{\ell}K \subset QD_2 K \to K \smsh K
 \end{CD}
 $$
 is an $S$-duality.

 Taking the adjoint of $f$, we obtain a map
 $$
 S^{d+\ell} \to \Sigma^\ell D_2^\ell K \, .
 $$
Suspend this map $(2j-\ell-1)$-times to obtain a map
$$
S^{d+2j-1} \to \Sigma^{2j-1} D_2^\ell K 
$$
Since $j = 2^{\phi(\ell-1)}$ we can apply Lemma \ref{useful} to compose this last
map with $\omega_\ell\:\Sigma^{2j-1}D_2^\ell K \to \Sigma^j K$ to get a map
$$
\gamma\:S^{d+2j-1} \to \Sigma^j K \, .
$$

Theorem \ref{first} is a direct consequence of the following.

\begin{prop} Let $Y = (\Sigma^j K) \cup_\gamma D^{d+2j}$ be the mapping cone of $\gamma$.
Then $Y$ is a Poincar\'e duality space of dimension $d+2j$ whose top cell splits off after $\ell$ suspensions.
\end{prop}

\begin{proof} Recall that the attaching map $\gamma$ is defined so as to factor
through $\omega_\ell$. Furthermore, by \ref{EP=0} and \ref{useful} we have a commutative diagram
$$
\xymatrix{
\Sigma^{2j-1} D_2^\ell K \ar[r]^{\omega_\ell} \ar[d] &\Sigma^j K \ar[d]^{E_\ell} \\
C \ar[r] & \Omega^\ell\Sigma^\ell \Sigma^j K
}
$$ 
with $C$ contractible. Hence the top cell of $Y$ splits off after $\ell$ suspensions.
It remains to show that $Y$ satisfies Poincar\'e duality.

The diagram together with $\gamma$ yields maps of pairs
$$
(D^{d+2j},S^{d+2j-1}) \to 
(\Omega^\ell\Sigma^\ell \Sigma^j K,\Sigma^j K) \overset{h_\ell}\to 
(\Omega^\ell\Sigma^\ell D_2^\ell(\Sigma^j K),C\Sigma^j K) \, .
$$
The induced map of quotients determines a homotopy class of map
$$
S^{d+2j} \to \Omega^\ell \Sigma^\ell D_2^\ell(\Sigma^j K) \, .
$$
Taking the adjoint and using the identification $D_2^\ell(\Sigma^j K) \simeq \Sigma^{2j}D_2^\ell (K)$
of \ref{James}, we obtain a map
$$
S^{d+2j+\ell} \to \Sigma^{2j+\ell}D_2^\ell K 
$$
which is just the $2j$-fold suspension of the map $f\: S^{d+\ell} \to \Sigma^\ell D_2 K$.
Since the transfer applied to $f$ is an $S$-duality map, 
it now follows by \ref{criterion} that $Y = (\Sigma^j K) \cup_\gamma D^{d+2j}$
is a Poincar\'e duality space.
\end{proof}

\section{Proof of Theorems \ref{quadratic} and \ref{quadratic-refinement}}

Theorem \ref{quadratic} is a direct consequence of Theorem \ref{quadratic-refinement}. 
\medskip

{\flushleft \it Proof that (2) implies (1) in Theorem \ref{quadratic-refinement}.} Suppose that $X = (\Sigma^j K) \cup_\beta D^n$ is a Poincar\'e space such that $\Sigma^j \beta$ is null homotopic, where $j = 2^{\phi(\ell-1)}$.
 Then we have a commutative diagram
$$
\xymatrix{
S^{n+2j-1} \ar[r]\ar[d] & \Sigma^j K \ar[d]\\
D^{n+2j} \ar[r]         & \Omega^{\ell}\Sigma^{\ell} \Sigma^j K \, .
}
$$
We can then take the composition of the maps of pairs
$$
(D^{n+2j},S^{n+2j-1}) \to (\Omega^{\ell}\Sigma^{\ell} \Sigma^j K,\Sigma^j K) 
\overset{h_\ell} \to 
(\Omega^{\ell}\Sigma^{\ell} D_2(\Sigma^j K),C\Sigma^j K)
$$
where $h_\ell$ is the Hopf invariant.

Taking the adjunction and quotients, we obtain a stable map
$$
f\:S^{n+2j+\ell} \to \Sigma^\ell D_2(\Sigma^j K) \, .
$$
As in the proof of \ref{criterion}, the composition given by 
$f$ followed by the transfer is a Spanier-Whitehead duality.  

Using \ref{James}, we have $D_2^\ell(\Sigma^j K) \simeq \Sigma^{2j} D_2^\ell K$
and this equivalence is compatible with transfers. Therefore, $f$ becomes identified
with a stable map
$$
S^n \stableto D^\ell_2K
$$
which is an $\ell$-refined quadratic duality.
\medskip

{\flushleft \it Proof that (1) implies (2).} Given $f\:S^d \stableto D_2^\ell K$,
let $j$ be large enough so that $\Sigma^{2c-1} f$ exists as an ordinary map
$$
S^{d+2c-1} \to \Sigma^{2c-1} D_2^\ell K \, .
$$
We can assume that $c = e2^{\phi(\ell-1)}$ for some $e\ge 1$.
Compose now with $\omega_\ell$ to obtain a map
$$
\gamma\:S^{d+2c-1} \to \Sigma^c K
$$
and form its mapping cone $Y = (\Sigma^c K) \cup_{\gamma} D^{d+2c}$. By 
\ref{criterion}, we infer that $Y$ is a Poincar\'e space whose top cell splits off 
after $\ell$-suspensions. Theorem \ref{first} shows that $Y$ is $j$-periodic
with $j = 2^{\phi(\ell-1)}$.

\section{Appendix: the stable Hopf invariant composition formula}

In this section we will recall Crabb's definition \cite[p.\ 62]{Crabb} of the stable Hopf invariant
and use it to derive the formula  for the Hopf invariant of a composition.
We are including this material for the sake of completeness (as far as I know, the composition formula does not appear anywhere in the literature).

For based $\Bbb Z_2$-CW complexes $A$ and $B$, let
$$
[A,B]^{\Bbb Z_2}
$$
denote the associated set of equivariant  homotopy classes of based $\Bbb Z_2$-maps. Let
$$
\{A,B\}^{\Bbb Z_2}
$$
be the colimit of $[S^W\smsh A,S^W\smsh B]^{\Bbb Z_2}$, where $W$ ranges over
a complete universe of representations of $\Bbb Z_2$. Then $\{A,B\}^{\Bbb Z_2}$
is an abelian group.

If $X$ and $Y$ is a based CW complexes, there is a canonically split short
exact sequence of abelian groups
$$
\begin{CD}
0 @>>> \{X,D_2Y\} @>\delta>> \{X,Y\smsh Y\}^{\Bbb Z_2} @>\rho >> \{X,Y\} @>>> 0
\end{CD}
$$
where the map $\delta$ is induced by the transfer and $\rho$ assigns to an equivariant map
$f\:S^W \smsh X \to S^W \smsh Y \smsh Y$ the map 
$$
f^{\Bbb Z_p}\:S^{W'} \smsh X \to S^{W'} \smsh Y
$$
given by taking fixed points. Here $W'$ is the fixed set of $W$, $X$ asnd $Y$ are given the trivial
action and $Y\smsh Y$ has the permutation action. The splitting for $\rho$ is given
by assigning to $X \stableto Y$ its composition with the reduced diagonal map
$\Delta_Y \: Y \to Y\smsh Y$.

The {\it stable Hopf invariant} of a stable map $f\: X \stableto Y$ is the unique
homotopy class $h(f)\: X \stableto D_2Y$ such that
$$
\delta(f) \,\, := \,\, (f\smsh f)\circ \Delta_X  \,\, - \,\,  \Delta_Y \circ f   \, .
$$

\begin{prop} Let $f\: X \stableto Y$ and $g\: Y \stableto Z$ be stable maps.
Then
$$
h(g\circ f) \,\, = \,\, D_2(g) \circ h(f)  \,\, +  \,\, h(g) \circ f\, .
$$
\end{prop} 

\begin{proof} Consider the following maps
\begin{enumerate}
\item $(g\smsh g)\circ (f\smsh f) \circ \Delta_X$,
\item $(g\smsh g) \circ \Delta_Y\circ f$,
\item $\Delta_Z \circ g\circ f$.
\end{enumerate}
Then
\begin{align*}
\delta(h(g\circ f)) \,\, &= \,\, (1) - (3)\, , \\
 &= \,\, (1) - (2) + (2) -  (3) \, ,\\
 & = \,\, \delta(D_2(g)\circ h(f))  +  \delta(h(g) \circ f)\, , \\
 & = \,\, \delta(D_2(g)\circ h(f))  +  h(g) \circ f)\, .
 \end{align*}
Now use the fact that $\delta$ is injective to complete the proof.
\end{proof}




\end{document}